\documentclass[a4paper, reqno, 11pt]{article}
\usepackage{amsmath, amsfonts, amsthm, amssymb}
\usepackage{graphicx}
\usepackage{float}
\usepackage{verbatim}
\usepackage{color}
\usepackage[sort]{cite}

\numberwithin{equation}{section}

\renewcommand{\geq}{\geqslant}
\renewcommand{\leq}{\leqslant}
\DeclareMathOperator{\bbR}{\mathbb{R}}
\DeclareMathOperator{\bbQ}{\mathbb{Q}}
\DeclareMathOperator{\bbN}{\mathbb{N}}
\DeclareMathOperator{\bbE}{\mathbb{E}}
\DeclareMathOperator{\bbP}{\mathbb{P}}

\DeclareMathOperator{\cB}{\mathcal{B}}
\DeclareMathOperator{\cL}{\mathcal{L}}
\DeclareMathOperator{\fM}{\mathfrak{M}}
\DeclareMathOperator{\cM}{\mathcal{M}}
\DeclareMathOperator{\fS}{\mathfrak{S}}
\DeclareMathOperator{\Erf}{Erf}

\DeclareMathOperator{\eps}{\varepsilon}
\DeclareMathOperator{\diam}{diam}

\DeclareMathOperator{\ess}{ess}
\DeclareMathOperator{\Var}{Var}

\allowdisplaybreaks

\newtheorem{theorem}{Theorem}[section]

\newtheorem{definition}[theorem]{Definition}
\newtheorem{corollary}[theorem]{Corollary}
\newtheorem{lemma}[theorem]{Lemma}

\title{On the Minkowski content of self-similar random homogeneous iterated function systems}
\author{Sascha Troscheit\footnote{The author was financially supported by Austrian Science Fund (FWF) Lise Meitner Senior Fellowship
M-2813.}\\\texttt{sascha.troscheit@oulu.fi}}
\begin{document}

\maketitle
\begin{center}\small
  \vskip-2em 
  Research Unit of Mathematical Sciences, P.O.~Box 8000, FI-90014 University of Oulu, Finland.
\end{center}

\begin{abstract}
  The Minkowski content of a compact set is a fine measure of its geometric scaling. For Lebesgue
  null sets it measures the decay of the Lebesgue measure of epsilon neighbourhoods of the set.
  It is well known that self-similar sets, satisfying reasonable separation conditions and non-log
  comensurable contraction ratios, have a well-defined Minkowski
  content. When dropping the contraction conditions, the more general notion of average Minkowski
  content still exists.
  For random recursive self-similar sets the Minkowski content also exists almost surely, whereas
  for random homogeneous self-similar sets it was recently shown by Z\"ahle that the Minkowski
  content exists in expectation.

  In this short note we show that the upper Minkowski content, as well as the upper average Minkowski
  content of random homogeneous self-similar sets is infinite, almost surely, answering a conjecture
  posed by Z\"ahle.
  Additionally, we show that in the random homogeneous equicontractive self-similar setting the 
  lower Minkowski content is zero and the lower average Minkowski content is also infinite.
  These results are in stark contrast to the random recursive model or the mean behaviour of random
homogeneous attractors.\end{abstract}

\vskip2em

\section{Introduction}
The $s$-dimensional Minkowski content of a compact set $K\subset \bbR^d$ is defined as the limit 
\begin{equation}\label{eq:ContentDefinition}
  \cM^s (K) = \lim_{\eps\to 0}\eps^{s-d}\cL^d(\langle K\rangle_{\eps}),
\end{equation}
where $\langle K \rangle_{\eps}=\{x\in\bbR^d: \inf_{y\in K}d(x,y) \leq \eps\}$ is the (closed)
$\eps$-neighbourhood of $K$. 
If the limits exist for all $s\geq 0$, there exists a critical exponent
$s\geq 0$ such that $\cM^t(K)=0$ for all $t<s$ and $\cM^t(K)=\infty$ for all
$t>s$. This critical exponent is known as the Minkowski dimension of $K$, which coincides with the
box-counting dimension of $K$.

Note that the limit in \eqref{eq:ContentDefinition} may not exist. Instead, we may take upper and
lower limits giving the notions of upper and lower Minkowski content (denoted by $\overline{\cM}^s$
and $\underline{\cM}^s$, respectively).
The upper and lower Minkowski contents have well-defined critical exponents that are referred to as
the upper and lower Minkowski dimension (or upper/lower box-counting dimension). If these critical
exponents coincide, we speak of the Minkowski dimension of $K$. 
However, even if the Minkowski dimension of $K$ exists and therefore is the critical exponent of
the Minkowski content, the limit in \eqref{eq:ContentDefinition} may still not exist. 
Its existence is therefore a measure of ``regularity'' and constitutes an interesting property of a
sets. We say that a set $K$ is Minkowski measurable (with dimension $s$) if $\cM^s(K)\in
(0,\infty)$.
For more background on the Minkowski content and its relation to Zeta functions, see
\cite{Lapidus16}.

\paragraph{Self-similar sets.}
Self-similar sets are compact sets that are invariant under a finite collection of contracting
similarities $f_1,\dots,f_n$. That is, for maps satisfying $|f_i(x)-f_i(y)|=r_i|x-y|$ for all
$x,y\in\bbR^d$, the associated self-similar set $F$ is the unique non-empty compact set that satisfies
\[
  F = \bigcup_{i=1}^n f_i(F).
\]
These self-similar sets are the quintessential fractal sets that are widely studied, especially
under assumptions that limit the overlaps $f_i(F)\cap f_j(F)$.
We refer the reader to \cite{Falconer97} and \cite{Falconer13} for an overview of dimension
theoretic properties.

Let $f_i:[0,1]\to[0,1]$ be finite collection of contracting similarities on the unit interval
indexed by $i=1,\dots,N$ with contraction rates $r_i \in(0,1)$. Assume that the maps satisfy the 
separation condition: $f_i([0,1])\cap f_j([0,1]) = \varnothing$ whenever $i\neq j$.
Under these assumptions, Falconer \cite{Falconer95} showed
the following dichotomy:
If  two of the similarities have log-incomensurable contraction rates, i.e.\ $\log r_i/\log r_j 
\notin \bbQ$ for some $i,j$ (the non-arithmetic case), then $F$ is Minkowski measurable. 
Otherwise (the arithmetic case), there exists a non-trivial periodic function
$g$, such that $|\eps^{s-d}\cL^s(\langle F\rangle_{\eps}) - g(-\log\eps)|\to 0$.
In particular, the Minkowski content does not exist but the lower and upper Minkowski content are
positive and finite.
These results extend naturally to self-similar sets in higher dimensions, see \cite[Corollary
7.6]{Falconer97} and hold under less restrictive separation conditions such as the open set
condition\footnote{A collection of contractions $(f_i)$ with attractor $F$ satisfies the \textbf{open set
  condition} if there exists a non-empty open set $O$ such that   
  $f_i(O)\subseteq O$ and $f_i(O)\cap f_j(O)=\varnothing$
whenever $i\neq j$.},
see Gatzouras \cite{Gatzouras00}.
Further generalisations can be made such as replacing the $d$-dimensional Lebesgue measure with
$n$-dimensional ``curvatures'' (or ``intrinsic volumes'') that capture the lower-order behaviour of
fractal sets, see Winter \cite{Winter08} for details.

While the Minkowski content does not exist in the arithmetic case,
the behaviour of
$\eps^{s-d}\cL^d(\langle F\rangle_{\eps})$ is still
very regular. To consolidate these behaviours,
a slightly weaker definition of content is required, the average Minkowski content.

\paragraph{The average Minkowski content.}
The average Minkowski content of a compact set $K$ is given by
\begin{equation}\label{eq:DefinitionAverage}
  \fM^s(K) = \lim_{\delta\to 0}\frac{1}{|\log \delta|}\int_{\delta}^1\eps^{s-d}\cL^{d}(\langle
  K\rangle_{\eps})\frac{1}{\eps}d\eps.
\end{equation}
The average Minkowski content is an averaging over the decay of the Lebesgue measure, that 
preserves the critical exponent of a set if it exists. That is, if the Minkowski dimension of a set
$K$ is $s>0$, then $\fM^t(K)=0$ for $t<s$ and $\fM^t(K) =\infty$ for $t>s$.
Should the small scale behaviour be periodic, as is the
case for some self-similar sets, the average Minkowski content still exists.
As with the Minkowski content, we write $\underline{\fM}^s$ and $\overline{\fM}^s$ for the upper and
lower average Minkowski content, obtained by taking the upper and lower limit in
\eqref{eq:DefinitionAverage}, respectively.

\paragraph{Random attractors and results.}
It is a general observation that introducing randomness can homogenise local structures. This is
especially visible in stochastically self-similar sets, first considered in \cite{Falconer86,
Graf87,Mauldin86}.
These sets satisfy a similar invariance to deterministic self-similar sets given by
\[
  F =_d \bigcup f_i (F)
\]
where the equality holds in distribution, and the sets as well as the maps are independent.
These sets are generally well-behaved, have positive and finite Hausdorff measure for an appropriate
gauge function \cite{Graf88} and are almost surely Minkowski measurable \cite{Gatzouras00}. We will
refer to this model of randomness as the random recursive model.

Another important model of stochastic self-similarity is that of random homogeneous sets, also known as
$1$-variable sets (named after the more general notion of $V$-variable sets \cite{Barnsley12}).
In this model, at every iteration step only one family of functions is chosen (independently of
other levels) and applied to every subset. This introduces geometric dependencies and was shown to
behave rather differently to the random recursive model. In particular, the almost sure behaviour is
determined by a ``geometric expectation'' ($\exp\bbE(\log(.))$) as opposed to the ``arithmetic
expectation'' ($\bbE$) for random recursive sets.
The Hausdorff and Minkowski dimension of random homogeneous attractors are
in general strictly smaller than their random recursive analogue \cite{Hambly92,Roy11,Troscheit17}
and there is no gauge function that gives positive and finite Hausdorff
measure for random homogeneous sets \cite{Troscheit21}.
In \cite{Zahle20}, Z\"ahle investigated the Minkowski and average Minkowski content in
the random homogeneous setting.
If one assumes a non-arithmetic condition, which is e.g.\ satisfied if there is positive probability
that an IFS is chosen with log-incomensurable contraction ratios, the Minkowski
content of the random attractor exists in expectation.
Z\"ahle further showed that the average Minkowski content exists in
expectation independent of the non-arithmetic condition.
We point out that these results are only achieved in expectation, and not almost surely.
In fact, the critical exponents no longer agree: the almost sure Minkowski dimension
of the random homogeneous attractor is strictly smaller than the expected Minkowski dimension of the
attractor.
For related work, see also Z\"ahle \cite{Zahle11} and Rataj, Winter, and Z\"ahle \cite{Rataj21}.
In \cite{Zahle20}, Z\"ahle further conjectured that the almost sure Minkowski content does not exist for random
homogeneous self-similar sets and in this article we show that this is indeed the case.

In particular we show, under mild separation assumptions, but no assumptions on the non-arithmetic
nature of the contractions, that the upper Minkowski and average 
Minkowski content of random homogeneous self-similar sets is infinite. 
In the special case where the random functions are equicontractive at each construction level, we
additionally show that the lower Minkowski content is zero and that the lower average Minkowski
content is infinite.

\section{Definitions and Results}
\subsection{Notation}
Let $0<r_{\min}<r_{\max}<1$ and let
$S_d$ be the set of contracting similarities $f:\bbR^d\to\bbR^d$ that map the closed unit ball
$K=B(0,1)$ into itself ($f(K)\subseteq K$) and have contraction ratio bounded above by $r_{\max}$
and below by $r_{\min}$.
Equip $S_d$
with the topology of pointwise convergence and write $\cB_d$ for its Borel $\sigma$-algebra. Let
$\Lambda =
\bigcup_{k=1}^\infty S_d^k$ and let $\cB_d^*=\{B\subset \Lambda : B\cap S_d^k \in (\cB_d)^k, \forall
k\in\bbN\}$ be the natural $\sigma$-algebra on $\Lambda$. 
Let $\bbP_1$ be a probability measure on $(\Lambda,\cB_d^*)$.

The product space $(\Omega,\cB,\bbP) = (\Lambda,\cB_d^*,\bbP_1)^{\bbN}$ where each
realisation 
\[
  \Omega\ni\omega = (\omega_1,\omega_2,\dots) =
  ((f_{\omega_1}^1,f_{\omega_1}^2,..,f_{\omega_1}^{N_{\omega_1}}),\;
  (f^1_{\omega_2},\dots,f^{N_{\omega_2}}_{\omega_2}),\dots)
\]
is a sequence of $N_{\omega_i}$ many similarities denoted by $f_{\omega_i}^j$ for
$j\in\{1,\dots,N_{\omega_i}\}$. We call $(\Omega,\cB,\bbP)$ a random iterated function system
(RIFS).
For convenience we will write $\Sigma_\lambda=\{1,\dots,N_{\lambda}\}$ and
$\Sigma(\omega)=\Sigma_{\omega_1}\times\Sigma_{\omega_2}\times \dots$ to refer to the indices of
the maps in $\lambda\in\Lambda$ and their infinite codings.
For $v\in\Sigma_\lambda$, we write $r_{\lambda}^v$ for the contraction ratio of $f_{\lambda}^v$,
i.e.\ $|f_\lambda^v(x) - f_{\lambda}^{v}(y)| = r_\lambda^v |x-y|$ for all $x,y\in\bbR^d$.

To define the homogeneous random attractor $F_\omega$ we define the projection
$\pi:\Omega\times\Sigma(\omega)\to\bbR^d$, given by
\[
  \pi((\omega,v)) = \lim_{n\to\infty} f_{\omega}^{v|_n}(0) = \lim_{n\to\infty}
  f_{\omega_1}^{v_1}\circ \dots \circ f_{\omega_n}^{v_n}(0).
\]
Since all $f_{\lambda}^w$ are strict contractions on a compact set, the limit is well-defined.
The attractor $F_\omega$ is then given by the projection of all words
\[
  F_\omega = \bigcup_{v\in\Sigma(\omega)}\pi( (\omega,v) ).
\]
Choosing $\omega$ with law $\bbP$, gives rise to the random attractor $F_{\omega}$.

The set may equivalently be defined as the $\limsup$ set of covers of increasing
levels.
We set 
\[
  F_\omega^n = \bigcup_{v\in\Sigma(\omega)}f_{\omega}^{v|_n}(K)
  =\bigcup_{v\in\Sigma(\omega)}f_{\omega_1}^{v_1}\circ \dots\circ f_{\omega_n}^{v_n}(K)
\]
Then $F_\omega \subseteq F_{\omega}^n$ for all $n\in\bbN$ since $f_{\lambda}^j(K)\subseteq K$.
Further, $d_H(F_\omega,F_\omega^n)\to 0$ as
$n\to\infty$, where $d_H$ is the Hausdorff distance. This gives the alternative definition
\[
  F_\omega = \lim_{n\to\infty}F_\omega^n = \bigcap_{n\in\bbN}F_\omega^n = \lim_{n\to\infty}
  \bigcap_{i=n}^\infty  \bigcup_{v\in\Sigma(\omega)}f_{\omega}^{v|_n}(K).
\]
To study its Minkowski content, we make
the following assumption on how images for distinct words are separated.
\begin{definition}
  We say that a RIFS $(\Omega,\cB, \bbP)$ satisfies the \textbf{uniform cylinder
  separation condition} if there exist $\gamma>0$ and $x_0\in\bbR^d$ such that for $\bbP$-almost all
  $\omega\in\Omega$ and for all $v,w\in\Sigma(\omega)$, $n,m\in\bbN$,
  \[
    |f_{\omega}^{v|_n}(x_0)-f_{\omega}^{w|_m}(x_0)| \geq \gamma \min\{ r_{\omega_1}^{v_1}\dots
      r_{\omega_n}^{v_n}, r_{\omega_{1}}^{w_1}\dots
    r_{\omega_m}^{w_m}\}
  \]
  whenever $v|_k\neq w|_k$ for $k=\min\{n,m\}$.
\end{definition}
Several common separation conditions such as the uniform strong separation
condition as well as the uniform open set condition\footnote{A RIFS $(\Omega,\cB,\bbP)$ 
  satisfies the uniform open set
  condition if there exists an non-empty open set $O$ such that $\bbP_1$-almost every IFS $\lambda\in\Lambda$
satisfies the open set condition with $O$.} satisfy this condition.
It is an adaptation of the  weak separation condition (as used in \cite{Angelevska20}) for random
sets. In the deterministic setting, the weak separation condition is an important generalisation of
the open set condition, and we refer the reader to \cite{Kaenmaki16} for a discussion.
In Lemma \ref{thm:separation} we show that the
uniform open set condition implies the uniform cylinder separation condition.

As a further assumption, we need to ensure that the random iterated function is indeed ``random''.
\begin{definition}
  We say that an RIFS $(\Omega,\cB,\bbP)$ is \textbf{almost deterministic} if there
  exists $s\in\bbR$ such that 
  \[
    \bbP_1\left\{\lambda\in\Lambda : \sum_{i=1}^{N_\lambda}(r_{\lambda}^{i})^s = 1 \right \} =1
  \]
  Conversely, an RIFS is \textbf{not almost deterministic} if there exists no such $s\in\bbR$.
\end{definition}

Therefore, assuming that our RIFS is not almost deterministic means we get a bona-fide random set
where coverings are (almost surely) not geometrically similar.
To avoid trivial singleton attractors, and to simplify calculations we will also make the assumption that 
\begin{equation}\label{eq:supAssumption}
  2\leq \underset{\lambda\sim\bbP_1}\ess\sup_{\lambda\in\Lambda} N_{\lambda}<\infty.
\end{equation}
Recall that the Minkowski dimension of all random homogeneous sets exists almost surely and
coincides with the almost sure Hausdorff dimension irrespective of overlap conditions, see
\cite{Troscheit17}. We refer to this almost sure value as the \textbf{essential Minkowski dimension}
of the RIFS $(\Omega,\cB,\bbP)$ and usually denote it by $s = \ess_{\omega\sim \mathbb{P}}\dim_M
F_\omega$.

Our main result for general self-similar random iterated function systems is
\begin{theorem}\label{thm:general}
  Let $F_\omega$ be the attractor of the self-similar random iterated function system
  $(\Omega,\cB,\bbP)$. Write $s$ for its essential Minkowski dimension.
  Assume that the RIFS satisfies the
  uniform cylinder separation condition, is not almost deterministic, and satisfies
  \eqref{eq:supAssumption}.  Then, almost surely,
  \[
    \overline{\cM}^s(F_\omega) = \infty
    \quad\text{and}\quad
    \overline{\fM}^s(F_\omega) = \infty
  \]
  In particular, the attractor $F_\omega$ is almost surely not Minkowski measurable and does not
  have finite average Minkowski content.
\end{theorem}

If we further restrict the random iterated function system to be equicontractive for every
$\lambda\in\Lambda$, we can say more on the Minkowski content and average Minkowski content.
\begin{theorem}\label{thm:equicontractive}
  Let $F_\omega$ be the attractor of the self-similar random iterated function system
  $(\Omega,\cB,\bbP)$. Write $s$ for its essential Minkowski dimension.
  Assume that the RIFS satisfies the
  uniform cylinder separation condition, is not almost deterministic, and satisfies
  \eqref{eq:supAssumption}.
  Assume further that for $\bbP_1$-almost all $\lambda\in\Lambda$ there exists $r_\lambda$  such that
  $r_\lambda^i=r_\lambda$ for all $1\leq i \leq N_\lambda$.
  Then, almost surely,
  \[
    \underline{\cM}^s (F_\omega) = 0
    \quad\text{and}\quad
    \underline{\fM}^s (F_\omega) =\infty.
  \]
\end{theorem}
Combining Theorems \ref{thm:general} and \ref{thm:equicontractive} we obtain
\begin{corollary}
  Let $F_\omega$ be the attractor of an RIFS satisfying the assumptions of Theorem \ref{thm:equicontractive}.
  The upper and lower Minkowski content of $F_\omega$ are almost surely as divergent as possible
  \[
    0=\underline{\cM}^s(F_\omega)<\overline{\cM}^s(F_\omega) = \infty
  \]
  whereas the average Minkowski content
  is infinite, 
  $\fM^s(F_\omega) = \infty$, almost surely.
\end{corollary}

Note that this is in stark contrast to the random recursive model where, almost surely, the
stochastically self-similar set is Minkowski measurable, \cite{Gatzouras00}.
Note also that this shows that the almost surely behaviour is drastically different to the
behaviour in expectation, as analysed by Z\"ahle \cite{Zahle20}.

\paragraph{$V$-variable sets and random set with a neck structure.}
Random homogeneous iterated function systems are a special case of the more general set up of
$V$-variable attractors, proposed by Barnsley et al.~\cite{Barnsley12}. This in turn can be
generalised to code trees with a neck structure, see e.g.\ \cite{Jarvenpaa17}. The defining feature
of these is a relaxing of the condition that all subtrees at a level have to be identical, as is
the case for random homogeneous (or $1$-variable) attractors. $V$-variable fractals are conditioned
to have at most $V$ different subtrees at every construction level, whereas code trees with necks are
those attractors where there are infinitely many levels (the necks) where all subtrees are
identical.

It is this recurrent structure that was used in \cite{Troscheit21} to show that the Hausdorff
measure cannot be positive and finite, regardless of gauge functions. 
It appears as though there are no barriers to extending the observations in this paper
to $V$-variable
attractors and code trees with necks, and we conjecture that they, too, have infinite upper Minkowski
and upper average Minkowski
content.

\section{Proofs}
As remarked above, the essential Hausdorff and Minkowski dimensions of 
homogeneous random self-similar sets coincide and are given by the unique $s$ for which
\[
  \bbE_{\bbP_1}\left(\log \sum_{i=1}^{N_\lambda} (r_{\lambda}^i)^s \right):=
  \int_{\Lambda}\log \sum_{i=1}^{N_\lambda} (r_{\lambda}^i)^s \; d\bbP_1(\lambda)=0,
\]
see
e.g.\ \cite{Hambly92,Troscheit17}.
To ease notation, we will refer to the Hutchinson sum above by 
\[
  \fS^s_\lambda : = \sum_{i=1}^{N_\lambda} (r_{\lambda}^i)^s
\]
If the random iterated function system is not almost deterministic, $\log\fS^s_{\lambda}$
is a random variable with mean $0$ and there exists positive probability that $\log\fS_\lambda^s
\neq 0$, i.e.\ it has positive variance. To show that its variance is also finite, consider
\[
  \bbE_{\bbP_1}\left( \left( \log \sum_{v=1}^{N_\lambda}(r_{\omega_1}^v)^s \right)^2 \right) 
  \leq  \max\left\{ (\log r_{\min}^s)^2, (\log(N r_{\max}^s))^2 \right\} <\infty. 
\]

Recall the Lyapunov Central Limit Theorem
(CLT) and the law of the iterated logarithm.
\begin{theorem}[Lyapunov Central Limit Theorem]
  Let $X_n$ be a sequence of square integrable random variables with mean $m_n$ and variance
  $v_n>0$.
  Assume that there exists $\delta>0$ such that
  \[
    \lim_{n\to\infty}\left(\sqrt{\sum_{i=1}^n
    v_i}\right)^{-(2+\delta)}\cdot\sum_{i=1}^n\bbE(|X_i-m_i|^{2+\delta}) =0.
  \]
  Then, 
  \[
    \frac{1}{\sqrt{\sum_{i=1}^nv_n}}\sum_{i=1}^n (X_i-m_i)
  \]
  converges in distribution to the normal distribution with mean $0$ and variance $1$.
\end{theorem}
For a proof and detailed discussion see, for example, \cite[\S28]{Bauer}.
Note that the theorem above makes the assumption that the variance is positive for all $n$. This can
without loss of generality be relaxed to $v_n =0$ for some $n$, so long as $\sum v_n \to \infty$.
This is because if $\nu_n=0$ then $X_n-m_n=0$ almost surely.

We will also need the law of the iterated logarithm (LIL).
\begin{theorem}[Wittmann Law of the Iterated Logarithm \cite{Wittmann85}]
  Let $X_n$ be a sequence of square integrable random variables with mean $m_n$ and variance
  $v_n>0$. Assume that the sequence satisfies
  \[
    \sum_{i=1}^{\infty}\frac{\bbE(|X_i|^p)}{(2 s_n\log \log s_n)^{p/2}} <\infty
  \]
  for some $2<p\leq 3$, where $s_n = \sum_{i=1}^n v_i$. Assume further that $s_n\to\infty$ and
  $\limsup_{n\to\infty}s_{n+1}/s_n < \infty$.
  Then,
  \[
    \limsup_{n\to\infty}\frac{\sum_{i=1}^n X_i}{\sqrt{2s_n\log\log s_n}} =1
  \]
  and
  \[
    \liminf_{n\to\infty}\frac{\sum_{i=1}^n X_i}{\sqrt{2s_n\log\log s_n}} =-1.
  \]
\end{theorem}
Again, we may let $v_n=0$ for some $n$, given that this does not affect the outcome of the sum. We
further note that for low $n$, the value of $\log\log n$ is not defined. Since we are only
interested in limits, we may assume $n$ is large enough such that this is well-defined.

We note that any sequence of random variables for which $X_n$ and $s_n$ are bounded immediately
satisfies the conditions of both theorems.

\subsection{Proof of Theorem~\ref{thm:general}}
To prove the Theorem~\ref{thm:general}, we first construct a random measure $\mu_\omega$ on the symbolic
space $\Sigma(\omega)$.

Recall that the Hutchinson sum satisfies
$\bbE(\log\fS_{\omega_1}^s)=0$, where $s$ is the essential Minkowski dimension.
The fact that $s>0$ follows directly from the assumptions that
$r_{\lambda}^i \geq r_{\min}$ and $\bbP_1(N_\lambda \geq2)>0$. The argument is standard and left to
the reader.
Note that the assumption that the RIFS is not almost deterministic
is equivalent, by definition, to $\fS^s_\lambda$ not being a constant $\bbP_1$-almost surely. 
Equivalently, the variation of $\fS_\lambda^s$ is positive.

To every letter $i\in \Sigma_\lambda$ we
associate probability $p_\lambda^i = (r_\lambda^{i})^s/\fS_\lambda^s$.
Then
\[
  \sum_{i\in\Sigma_\lambda} p_{\lambda}^i=
  \sum_{i\in\Sigma_{\lambda}}\frac{(r_{\lambda}^i)^s}{\fS_\lambda^s} = 1.
\]
Hence, the measure induced by setting $\mu_\omega([v|_n]) = \prod_{i=1}^n p_{\omega_i}^{v_i}$ for
cylinder $[v|_n] = \{w\in\Sigma(\omega) : (\forall 1\leq i\leq n)w_i = v_i\}$ is a
bona fide probability measure on $\Sigma(\omega)$.

Write 
\[
  \overline{r}_\lambda = \exp\sum_{i=1}^{N_\lambda} p_{\lambda}^i \log r_{\lambda}^i
\]
and 
\[
  v_\lambda = \sum_{i=1}^{N_\lambda} p_{\lambda}^{i}(\log r_{\lambda}^i -
  \log\overline{r}_\lambda)^2.
\]
The quantity $\overline{r}_\lambda$ is the geometric mean of the contraction rates with respect to
choosing letters $i\in\Sigma_\lambda$ with probabilities $\{p_{\lambda}^{i}\}$. Observe that 
$\log\overline{r}_\lambda$ is the (arithmetic) mean of the logarithms of the contraction rates.
The second quantity $v_\lambda$ then denotes the variance of the logarithm of
the contraction rate with respect to the same measure on $\Sigma_\lambda$.
Note that $v_{\lambda}$ may be zero. However, because we are considering systems which
are not almost deterministic, there exists positive probability that $v_\lambda>0$. This further
implies that for generic $\omega\in\Omega$ the sum $\sum_{i=1}^n v_{\omega_i}$ eventually grows faster than
$\eta n$ for some constant $\eta>0$.
The boundedness of the contraction ratios further imply
\[
  \lim_{k\to\infty}\frac{1}{(\sum_{k=1}^n v_{\omega_i})^{3}}\sum_{k=1}^n\sum_{i=1}^{N_\lambda} p_{\lambda}^{i}|\log r_{\lambda}^i -
  \log\overline{r}_\lambda|^3 = 0
\]
for generic $\omega\in\Omega$ and so the sum $\sum_{k=1}^n \log r_{\omega_k}^{v_n}$ satisfies the
Lyapunov Central Limit Theorem with respect to $\mu_\omega$.

Now write $M_n$ for the set
\[
  M_n = \left\{v\in\Sigma(\omega) : \prod_{i=1}^n r_{\omega_i}^{v_i} \in \left[ \prod_{i=1}^n
      \overline{r}_{\omega_i} \cdot e^{-\sqrt{\sum_{i=1}^n v_{\omega_i}}},\prod_{i=1}^n
  \overline{r}_{\omega_i}\cdot e^{\sqrt{\sum_{i=1}^n v_{\omega_i}}} \right]\right\}
\]
noting that
\[
  M_n = \left\{v\in\Sigma(\omega) : \frac{\sum_{i=1}^n ( \log r_{\omega_i}^{v_i}-\log
  \overline{r}_{\omega_i})}{\sqrt{\sum_{i=1}^n v_{\omega_i}}} \in [-1,1] \right\}.
\]
By the Lyapunov Central Limit Theorem, 
\[
  \mu_\omega(M_n) \to (2\pi)^{-1/2}\int_{-1}^{1}\exp(-x^2/2)dx=:q
  \quad\text{as}\quad n\to\infty.
\]
Hence, for large enough $n$, we have
\[
  q/2 \leq \mu_{\omega}(M_n) = \sum_{v\in M_n}\mu_\omega([v|_n]) = \sum_{v\in
  M_n}\frac{(r_{\omega_1}^{v_1}\dots
  r_{\omega_n}^{v_n})^s}{\fS_{\omega_1}^{s}\dots\fS_{\omega_n}^{s}}
  \leq \# M_n \frac{(\overline{r}_{\omega_1}\dots
  \overline{r}_{\omega_n})^s}{\fS_{\omega_1}^{s}\dots\fS_{\omega_n}^{s}}e^{\sqrt{\sum_{i=1}^{n}v_{\omega_i}}}
\]
and so
\[
  \# M_n \geq \frac{q}{2}\frac{ \fS_{\omega_1}^{s}\dots\fS_{\omega_n}^{s}}{(\overline{r}_{\omega_1}\dots
  \overline{r}_{\omega_n})^s}\cdot \exp\left[-\left(\sum_{i=1}^{n}v_{\omega_i} \right)^{1/2}\right]
\]
Recall that we are considering generic $\omega\in\Omega$ and that $\sum_{i=1}^{n}v_{\omega_i} \geq
\eta n$
for some $\eta>0$ and large enough $n\in\bbN$.
It is straightforward to show that all assumptions in the law of the iterated logarithm are satisfied for
generic $\omega\in\Omega$. Thus there exists a constant $C>0$ and a subsequence $n_k$ such that
$\fS_{\omega_{n_k}}^s \geq
\exp(C \sqrt{n_k\log\log n_k})$.
As a consequence, 
\[
  \# M_{n_k} \geq C' \exp\left(C\sqrt{n_k\log\log n_k}-\sqrt{\eta\, n_k}\right)(\overline{r}_{\omega_1}\dots
  \overline{r}_{\omega_{n_k}})^{-s}.
\]

We can use the uniform cylinder separation condition to obtain a lower bound on the cardinality of a
sufficiently separated set.
\begin{lemma}\label{thm:separationLemma}
  Let $(\Omega,\cB, \bbP)$ be a RIFS as in Theorem \ref{thm:general}.
  Then there exists a constant $C_d>0$ such that for $\bbP$-almost all $\omega\in\Omega$ the
  following holds.
  Let $A$ be a finite collection finite length words that are not prefixes of each other, i.e.\ 
  \[
    A\subseteq \bigcup_{n=1}^\infty \Sigma_{\omega_1}\times \dots \times\Sigma_{\omega_n}
    \;\;\text{ with distinct }\;\;
    v,w\in A \Rightarrow v|_{n}\neq w|_{n}, \; n=\min\{|v|,|w|\}.
  \]
  Then there exists a $(\gamma \min_{v\in A}r_{\omega}^{v})$-separated set $F_A\subset F_{\omega}$
  with $\#F_A \geq C_d\#A$.
\end{lemma}
\begin{proof}
  Let $A$ be given and assume that $\omega\in\Omega$ is generic.
  We define $A^*$ by considering descendants of $A$ such that the associated contraction is
  comparable to the minimal contraction in $A$. Concretely, writing $1^k$ for the word
  $(1,1,\dots,1)$ of length $k$ and $\underline r_A = \min_{v\in A}r_{\omega}^{v}$,
  \begin{multline*}
    A^* = \{ v1^k \in \Sigma_{\omega_1}\times \dots \times \Sigma_{\omega_{|v|+k}} \;:\; v\in A
      \;\;\text{ and }\\ r_{\omega_1}^{v_1}\dots r_{\omega_{|v|}}^{v_{|v|}}r_{\omega_{|v|+1}}^1 \dots
      r_{\omega_{|v|+k+1}}^1 < \underline r_A \leq r_{\omega_1}^{v_1}\dots r_{\omega_{|v|}}^{v_{|v|}}r_{\omega_{|v|+1}}^1 \dots
    r_{\omega_{|v|+k}}^1\}
  \end{multline*}
  We note that since the contractions are uniform, there must exist such $k\geq 0$ for every
  $v\in A$. Further, since elements in $A$ are not prefixes of each other, $A^*$ must also have this
  property and $\#A^* =\#A$. Additionally, the contractions are uniformly bounded from below, which
  implies that the contraction
  ratios satisfy
  \[
    \underline{r}_A \leq r_{\omega}^w \leq \frac{\underline{r}_A}{r_{\min}}.
  \]
  for all $w\in A^*$.

  The uniform cylinder separation
  condition implies that there exists $x_0\in\bbR^d$ such that 
  \[
    |f_{\omega}^{v}(x_0)-f_{\omega}^{w}(x_0)| \geq \gamma \underline{r}_A
  \]
  where $v,w\in A^*$. Therefore the balls
  $B(f_{\omega}^v(x_0),\gamma\underline{r}_A/2)$ are pairwise disjoint.

  Recall that by assumption $F_\omega \subset K=B(0,1)$. Therefore, $F_\omega \subset
  B(x_0,1+|x_0|)$ and for all $v\in A^*$,
  \[
    F_{\omega}\cap B\left(f_{\omega}^v(x_0),\; \frac{\underline{r}_A}{r_{\min}}(1+|x_0|)\right)
    \neq\varnothing.
  \]
  We can use a volume argument to obtain an upper bound on how many of those balls may intersect
  $F_\omega$, hence giving an upper bound to the cardinality of $A^*$.
  Let $A^*(v) = \{w\in A^* : |f_\omega^v(x_0) - f_\omega^v(x_0)| \leq 3
  (\underline{r}_A/r_{\min})(1+x_0)\}$  
  Then,
  \begin{align*}
    V_d \left( 4\frac{\underline{r}_A}{r_{\min}}(1+|x_0|) \right)^d
    &=
    \cL^d\left(  B\left(f_{\omega}^v(x_0),\; 4 \frac{\underline{r}_A}{r_{\min}}(1+|x_0|)\right)
    \right)\\
    &\geq \cL^d\left( \bigcup_{w\in A^*(v)} B\left( f_{\omega}^{w}(x_0),\; \gamma \underline{r}_A
    \right) \right)\\
    &=\sum_{w\in A^*(v)} \cL^d\left( B\left( f_{\omega}^{w}(x_0),\; \gamma \underline{r}_A
    \right) \right)\\
    &=\#A^*(v)V_d (\gamma \underline{r}_A)^d
  \end{align*}
  and so 
  \[
    \# A^*(v) \leq \left( 4\frac{1+|x_0|}{r_{\min} \gamma} \right)^d=:C_0.
  \]
  Therefore, we conclude that there exists a subset of $F_{\omega}$ with cardinality $\#A /C_0$
  consisting of points separated by $(\underline{r}_A/r_{\min})(1+|x_0|) \geq \gamma
  \underline{r}_A$, as claimed.
\end{proof}

By the uniform cylinder separation condition and Lemma \ref{thm:separationLemma}, there are at least
$C_0\#M_{n_k}$
many elements in $F_\omega$ that are $\gamma'  \prod_{i=1}^{n_k}
\overline{r}_{\omega_i}\cdot e^{-\sqrt{\sum_{i=1}^{n_k} v_{\omega_i}}}$ separated, since elements in
$M_{n_k}$ are distinct.
Hence, setting $\eps_{n_k} =\gamma'  \prod_{i=1}^{n_k}
\overline{r}_{\omega_i} \cdot e^{-\sqrt{\sum_{i=1}^{n_k} v_{\omega_i}}}$
we can find a lower bound by finding the Lebesgue measure of a disjoint union,
\begin{align*}
  \eps_{{n_k}}^{s-d}\cL^{d}(\langle F_{\omega}\rangle_{\eps_{n_k}}) &\geq
  \eps_{{n_k}}^{s-d}\cL^{d}\left( \bigcup_{v\in M_{n_k}}B(f_{\omega}^{v|_{n_k}}(0),\eps_{{n_k}})\right)
  \geq \eps_{n_k}^{s-d}V_d \eps_{{n_k}}^{d}C_0\#M_{n_k}\\
  & \geq \eps_{{n_k}}^s V_d C_0 C' \exp\left( C\sqrt{{n_k}\log\log {n_k}}-\gamma\sqrt{{n_k}}
  \right)(\overline{r}_{\omega_1}\dots\overline{r}_{\omega_{n_k}})^{-s}\\
  &\geq \gamma'^s V_d C_0C'\frac{1}{2^s}e^{-\gamma(1+s)\sqrt{{n_k}}}e^{C\sqrt{{n_k}\log\log {n_k}}}\\
  &\geq C'' e^{(C/2)\sqrt{{n_k}\log\log {n_k}}}.
\end{align*}
Hence, 
\begin{equation}\label{eq:contentExplosion}
  \eps_{{n_k}}^{s-d}\cL^{d}(\langle F_{\omega}\rangle_{\eps_{n_k}}) \to\infty
  \quad\text{along the subsequence}\quad {n_k}\to\infty.
\end{equation}
This shows that $\overline{\cM}^s(F_\omega) = \infty$.

\vskip.5em
We now show that the upper average Minkowski content also diverges.
We define $\delta_{n_k} = \gamma'  \prod_{i=1}^{n_k}
\overline{r}_{\omega_i} \cdot e^{-\sqrt{\sum_{i=1}^{n_k} v_{\omega_i}}}$. Then,
\begin{align*}
  &\frac{1}{|\log \delta_{n_k}/2|} \int_{\delta_{n_k}/2}^1 \eps^{s-d}\cL^d(\langle
  F_\omega\rangle_{\eps})\frac{1}{\eps}d\eps\\
  &\geq
  \frac{1}{|\log \delta_{n_k}/2|}\int_{\delta_{n_k}/2}^{\delta_{n_k}} (\delta_{n_k})^{s-d}\cL^{d}\left(
  \bigcup_{v\in M_{n_k}}B(f_{\omega}^{v|_{n_k}}(0),\delta_{{n_k}}/2)\right)\frac{1}{\eps}d\eps\\
  &\geq
  \frac{1}{|\log\delta_{n_k}/2|}(\delta_{{n_k}})^{s-d}V_d(\delta_{{n_k}}/2)^d
  \#M_{n_k}\frac{\delta_{n_k}-\delta_{n_k}/2}{\delta_{n_k}}\\
  &\geq C_1 \frac{1}{{n_k}}(\delta_{n_k})^s\#M_{n_k}
  \geq C_2 \frac{1}{{n_k}}e^{(C/2)\sqrt{{n_k}\log\log {n_k}}},
\end{align*}
where $n_k$ is the same subsequence as in \eqref{eq:contentExplosion}.
We conclude that $\overline{\fM}^s(F_\omega)=\infty$, proving our claim.\qed

\subsection{Proof of Theorem~\ref{thm:equicontractive}}
By assumption there
exist $r_\lambda$ such that $r_\lambda^v = r_\lambda$ for all $v\in\Sigma_\lambda$.
This greatly simplifies the expression for the number and size
of covering sets of the attractor. Indeed, the $k$ level set $F^k_\omega$ is a cover
consisting of exactly $N_{\omega_1}N_{\omega_2}\dots N_{\omega_k}$ balls of diameter exactly
$r_{\omega_1}r_{\omega_2}\dots r_{\omega_k}$.
By Lemma \ref{thm:separationLemma} this also means that $F_\omega$ contains at least
$C_0 N_{\omega_1}N_{\omega_2}\dots N_{\omega_k}$ many points separated by $\gamma' r_{\omega_1}r_{\omega_2}\dots
r_{\omega_k}$.
The Hutchinson sum reduces to $\fS^s_\lambda = N_\lambda (r_{\lambda})^s$ in the
equicontractive setting.

Since the sizes of $n$ level cylinders are the same, we can get improved approximations for the Lebesgue measure of
$\cL^d(\langle F_\omega\rangle_{\eps})$.
Fix a realisation $\omega\in\Omega$ and size $0<\eps<1$. Set $n$ such that
$r_{\omega_1}r_{\omega_2}\dots r_{\omega_n} < \eps \leq r_{\omega_1}\dots
r_{\omega_{n-1}}$.

Recall that by definition 
\[
  F_\omega = \bigcap_{k=1}^\infty F^k_\omega \subseteq F^{m}_\omega \quad \text{ and so } \quad
  \langle F_\omega\rangle_{\eps} \subseteq \langle F^m_\omega\rangle_{\eps}
\]
for all $m\in\bbN$.
Since $F^n_\omega$ consists of $N_{\omega_1}\dots N_{\omega_n}$ (possibly overlapping) images of
the unit ball $K$,
\begin{align*}
  \cL^d(\langle F_\omega \rangle_{\eps}) &\leq \cL^d(\langle F^n_\omega\rangle_{\eps})
  \leq \cL^d\left( \bigcup_{i=1}^{N_{\omega_1}\dots N_{\omega_n}}B(x_i,r_{\omega_1}\dots
  r_{\omega_n}+\eps )\right)\\
  &= N_{\omega_1}\dots N_{\omega_n}\cdot V_d\cdot (r_{\omega_1}\dots r_{\omega_n} +\eps)^d\\
  &\leq N_{\omega_1}\dots N_{\omega_n}\cdot V_d\cdot (r_{\omega_1}\dots r_{\omega_n}
  +r_{\omega_1}\dots r_{\omega_{n-1}})^d\\
  &\leq \left( 1+\frac{1}{r_{\min}} \right)^d V_d \cdot
  \prod_{i=1}^n r_{\omega_i}^d N_{\omega_i}.
\end{align*}
Conversely $F_\omega$ contains $C_0 N_{\omega_1}\dots N_{\omega_k}$ many points that are separated by
$\gamma' r_{\omega_1}\dots r_{\omega_k}$ for some uniform $\gamma'\leq 1$.
Hence, for $\eps \leq \gamma' r_{\omega_1} \dots r_{\omega_k}$, 
\[
  \cL^d(\langle F_\omega\rangle_{\eps}) \geq C_0 N_{\omega_1}\dots N_{\omega_k} V_d \eps^d.
\]
In particular, for $n$ such that $r_{\omega_1}\dots r_{\omega_n}< \eps \leq r_{\omega_1}\dots
r_{\omega_{n-1}}$, we let $k\leq n-2$ be the largest integer such that $k\leq n-\lceil\log(\gamma')/\log(r_{\max})\rceil-1$. Then,
\[
  \frac{r_{\omega_1}\dots
  r_{\omega_{n-1}}}{r_{\omega_1}\dots r_{\omega_{k}}}=
  r_{\omega_{k+1}}\dots r_{\omega_{n-1}} \leq r_{\max}^{n-k-1}\leq\gamma'
\]
and so
\[
  \eps \leq r_{\omega_1}\dots r_{\omega_{n-1}} \leq \gamma 'r_{\omega_1}\dots r_{\omega_k}
\]
as required.
This gives
\begin{align*}
  \cL^d(\langle F_\omega\rangle_{\eps}) &\geq C_0  N_{\omega_1}\dots N_{\omega_k} V_d \eps^d\\
  &\geq C_0 N_{\omega_1}\dots N_{\omega_k}V_d (r_{\omega_1}\dots r_{\omega_{n-1}})^d\\
  &\geq C_0 V_d\left( \ess \sup N_{\lambda} \right)^{-(n-k-1)} \prod_{i=1}^n r_{\omega_i}^d N_{\omega_i}.
\end{align*}
Thus, almost surely,
there exists a universal constant $C_1>0$ such that
\[
  \frac{1}{C_1}\prod_{i=1}^{n}r_{\omega_i}^d N_{\omega_i} \leq \cL^d(\langle
  F_\omega\rangle_{\eps}) \leq C_1 \prod_{i=1}^{n}r_{\omega_i}^d N_{\omega_i}
  \quad
  \text{for $n$ s.t.}
  \quad
  \prod_{i=1}^n r_{\omega_i} < \eps \leq \prod_{i=1}^{n-1} r_{\omega_i}.
\]
Equivalently\footnote{Let $a(x)=a(x,\omega)$ and $b(x)=b(x,\omega)$ be (random) functions. We write $a(x)\approx b(x)$
  if there exists a constant $C>0$ independent of $x$ and $\omega$,
such that $0<1/C < a(x)/b(x) <C < \infty$ for all $x$ and almost all $\omega$.},
\[
  \cL^d(\langle F_\omega\rangle_{\eps})\approx \prod_{i=1}^{n}r_{\omega_i}^d N_{\omega_i} = \exp
  \sum_{i=1}^{n}\log\fS^s_{\omega_i}
  \quad
  \text{for $n$ s.t.}
  \quad
  \eps \approx \prod_{i=1}^n r_{\omega_i}.
\]
Again, let $s>0$ be the unique exponent such that
$\bbE_{\bbP_1}(\log\fS^s_\lambda)=0$ and recall that the RIFS is not almost deterministic. In
particular this means that $v=\Var(\log\fS_{\lambda}^s)>0$. Hence
$\sum_{i=1}^{n}\log\fS_{\omega_i}^s$ is a symmetric random walk. It immediately follows that
\[
  -\infty =\liminf_{n\to \infty}\sum_{i=1}^{n}\log\fS_{\omega_i}^s
  < \limsup_{n\to\infty}\sum_{i=1}^n \log\fS_{\omega_i}^s = \infty.
\]
and so
\[
  \overline{\cM}^s =\limsup_{\eps\to0} \eps^{s-d}\cL^d(\langle F_\omega\rangle_{\eps}) =\infty\quad\text{and}\quad
  \underline{\cM}^s=\liminf_{\eps\to0} \eps^{s-d}\cL^d(\langle F_\omega\rangle_{\eps}) = 0,
\]
proving the first conclusion.

To show that the lower average Minkowski content also diverges to infinity we establish
that the average Minkowski content behaves like the arithmetic average of the content at geometric scales. 
Let $\delta>0$ be given and set $k$ such that $\prod_{i=1}^k r_{\omega_i}\approx
\delta$. Then,
\begin{align*}
  \frac{1}{|\log \delta|} \int_{\delta}^1 \eps^{s-d}\cL^d(\langle
  F_\omega\rangle_{\eps})\frac{1}{\eps}d\eps
  &= \frac{1}{|\log\delta|} \sum_{n=1}^k\int_{\prod_{i=1}^n
  r_{\omega_i}}^{\prod_{i=1}^{n-1} r_{\omega_i}} \eps^{s-d}\cL^d(\langle
  F_\omega\rangle_{\eps})\frac{1}{\eps}d\eps\\
  &\approx \frac{-1}{\log\prod_{i=1}^k r_{\omega_i}} \sum_{n=1}^k\exp\left( \sum_{i=1}^n \log
  \fS_{\omega_i}^s \right),
\end{align*}
where the last line follows as $\eps \approx \prod_{i=1}^n r_{\omega_i} \approx \prod_{i=1}^n
r_{\omega_i}-\prod_{i=1}^{n-1} r_{\omega_i}$.

We can further bound the integral by noting that $r_{\min}\leq r_{\omega_i} \leq r_{\max}$ and we
have
\begin{equation}\label{eq:approx}
  \frac{1}{|\log \delta|} \int_{\delta}^1 \eps^{s-d}\cL^d(\langle
  F_\omega\rangle_{\eps})\frac{1}{\eps}d\eps
  \;\approx\;
  \frac{1}{k} \sum_{n=1}^k\exp\left( \sum_{i=1}^n \log
  \fS_{\omega_i}^s \right).
\end{equation}
This shows again that the upper limit is infinite, since by the central limit theorem, there
exist infinitely many $k$ such that $\sum_{i=1}^k \log\fS_{\omega_i}^s > \sqrt{k}$.

To show that the lower limit is also infinite requires a little more effort and we need the
following lemma on random walks.
\begin{lemma}\label{thm:frequencyv2}
  Let $0<t<\tfrac12$.
  Write $W_n = \sum_{i=1}^n X_i$ for a random walk with i.i.d.\ increments. Assume that
  $\bbE(X_i)=0$ and $0<\Var(X_i)<\infty$. 
  Let $n_m$ be the unique integer sastifying $e^{m} \leq n_m < e^{m}+1$.
  Then, almost surely, there exists (random) $m_0\in\bbN$ such that for all $m\geq m_0$
  there exists $k\in[n_m,n_{m+1}-1]$ with
  $W_k > k^t$.
\end{lemma}
\begin{proof}
  We estimate 
  \begin{align*}
    &\bbP(W_k \leq k^t : \forall k\in[n_m,n_{m+1}-1])\\
    &=\bbP\left(W_k \leq k^t : \forall k\in[n_m,n_{m+1}-1]\;|\;W_{n_m} >  n_m^t\right)
    \cdot\bbP(W_{n_m} >  n_m^t)
    \\
    &\hspace{2em}+\bbP\left(W_k \leq k^t : \forall k\in[n_m,n_{m+1}-1]\;|\;W_{n_m}\leq n_m^t\right)
    \cdot\bbP(W_{n_m} \leq n_m^t)
    \\
    &\leq
    \bbP\left(W_k\leq k^t : \forall k\in[n_m,n_{m+1}-1]\;|\;W_{n_m}\leq   n_m^t\right)
    \\
    &\leq
    \bbP\left(W_k\leq (n_{m+1}-1)^t-n_m^t : \forall k\in[0,n_{m+1}-1-n_m]\right)
    \\
    &=
    1-\bbP\bigg( \sup \{W_k : {k\in[0,n_{m+1}-1-n_m]}\}\; > \;
    \left( n_{m+1}-1 \right)^t-n_m^t  \bigg)
    \intertext{and using the reflection principle gives}
    &=1-2 \bbP\bigg(W_{n_{m+1}-1-n_m}\; > \;
    \left( n_{m+1}-1 \right)^t-n_m^t  \bigg)
    \\
    &=\bbP\bigg(\big\lvert W_{n_{m+1}-1-n_m}\big\rvert \leq \left( n_{m+1}-1 \right)^t -n_m^t \bigg)
    \\
    &\leq
    \frac{2}{\sqrt{2\pi v_m}}
    \int_{-\left( n_{m+1}-1 \right)^t +n_m^t}^{\left( n_{m+1}-1 \right)^t -n_m^t}
    \exp\left( -\tfrac{x^2}{2 v_m} \right)dx 
    = 
    2 \Erf\left( \tfrac{\left( n_{m+1}-1 \right)^t -n_m^t}{\sqrt{2 v_m}} \right),
    \intertext{where $v_m=(n_{m+1}-1-n_m)\Var(X)$ and $m$ is assumed sufficiently large for the
    Gaussian approximation to hold. Then, using Taylor series,}
    &= \frac{2\sqrt{2}\left(\left( n_{m+1}-1 \right)^t -n_m^t\right)}{\sqrt{\pi v_m}} +O\left(\frac{\left(
    n_{m+1}-1 \right)^t -n_m^t}{\sqrt{2 v_m}}\right)^3
    \\
    &\leq \frac{4}{\sqrt{\pi\Var(X)}}\cdot
    \frac{\left( n_{m+1}-1 \right)^t -n_m^t}
    {\sqrt{n_{m+1}-1-n_m}}
    \leq \frac{4}{\sqrt{\pi\Var(X)}}\frac{e^{tm}(e^t-1)}{\sqrt{e^m(e-2e^{-m})}}
    \\[.5em]
    &\leq C e^{-(\tfrac12-t)m}
  \end{align*}
  for some uniform $C>0$. 
  Thus, the probability that $W_k$ does not exceed $k^t$ in $[n_{m+1},n_m-1]$ is summable in
  $m$. Hence, by the Borel-Cantelli lemma there are only finitely such $m$. This proves the lemma.
\end{proof}

We now show that the right hand side of \eqref{eq:approx} diverges to infinity using Lemma
\ref{thm:frequencyv2}.
Fix a generic $\omega\in\Omega$ and
let $t<1/2$. 
Let $n_m$ be the unique integer satisfying $e^m \leq n_m< e^{m}+1$. Applying Lemma
\ref{thm:frequencyv2}, we get
\begin{equation}\label{eq:prelim}
  \sum_{k=n_m}^{n_{m+1}-1}
  \log\fS_{\omega_k}^s \geq \sup_{k\in[n_m,n_{m+1}-1]}k_m^{t} \geq n_m^t\geq e^{mt}
\end{equation}
for all large enough $m\in \bbN$.
Let $k\in \bbN$ and let $m$ be such that $k\in [n_m,n_{m+1}-1]$. Then, using \eqref{eq:prelim},
\begin{align*}
  \frac{1}{k}\sum_{j=1}^{k}\exp\left( \sum_{i=1}^j \log\fS_{\omega_i}^s \right)
  &\geq \frac{1}{n_{m+1}-1}\sum_{j=1}^{n_{m}}\exp\left( \sum_{i=1}^j \log\fS_{\omega_i}^s \right)
  \\
  &\geq
  \frac{1}{n_{m+1}-1}\exp\left( \sum_{i=n_{m-1}}^{n_m-1} \log\fS_{\omega_i}^s \right)
  \\
  &\geq
  e^{-(m+1)}\exp\left( e^{mt} \right) \to\infty
\end{align*}
as $m\to\infty$.
But then, using \eqref{eq:approx},
\[
  \underline{\fM}^s(F_{\omega}) = \liminf_{\delta\to0} \frac{1}{|\log\delta|}\int_\delta^1 \eps^{s-d}\cL^d(\langle
  F_\omega\rangle_{\eps})\frac{1}{\eps}d\eps = \infty,
\]
showing that $\fM^s(F_\omega) = \infty$ almost surely.
\qed

\subsection{Separation conditions}
In this last section we prove that the uniform open set condition implies the uniform cylinder
separation condition.
\begin{lemma}\label{thm:separation}
  Let $(\Omega,\cB,\bbP)$ be a RIFS that satisfies the uniform open set condition. Then
  $(\Omega,\cB,\bbP)$ satisfies the uniform cylinder separation condition.
\end{lemma}
\begin{proof}
  Let $O$ be the set guaranteed by the uniform strong open set condition and let $x_0\in O$ and
  $\rho>0$ be small enough such that the closed ball $B(x,\rho)$ is contained in $O$.
  Let $v,w\in\Sigma(\omega)$ be such that $v|_k \neq w|_k$ but $v|_{k-1} = w|_{k-1}$ for some $k\in
  \bbN$.
  By the uniform open set condition, $f_{\omega}^{v|_n}(O)\cap f_{\omega}^{w|_m}(O)=\varnothing$ for all $m,n\geq
  k$ since $v_k \neq w_k$. Now, $B(x,\rho)\subset O$ and we further have
  $f_{\omega}^{v|_n}(B(x_0,\rho))\cap
  f_{\omega}^{w|_m}(B(x_0,\rho))=\varnothing$.
  This implies that 
  \begin{align*}
    |f_{\omega}^{v|_n}(x_0)-f_{\omega}^{w|_m}(x_0)|
    &\geq \tfrac12\diam(f_{\omega}^{v|_n}(B(x_0,\rho))) +
    \tfrac12\diam(f_{\omega}^{w|_m}(B(x_0,\rho)))
    \\
    &=
    \tfrac12 \rho \left( r_{\omega_1}^{v_1}\dots r_{\omega_n}^{v_n} + r_{\omega_1}^{w_1}\dots
    r_{\omega_m}^{w_m} \right)
    \\
    &\geq 
    \tfrac12 \rho \min \left\{ r_{\omega_1}^{v_1}\dots r_{\omega_n}^{v_n} , r_{\omega_1}^{w_1}\dots
    r_{\omega_m}^{w_m} \right\}
  \end{align*}
  and we see that the uniform cylinder separation condition holds for $\gamma = \tfrac12 \rho$.
\end{proof}

\subsection*{Acknowledgements}
The author is grateful to Martina Z\"ahle for bringing the question to the author's attention. The
author also thanks Martina Z\"ahle for comments on an earlier version of this manuscript and the
anonymous referee for their extensive comments and suggestions.

\end{document}